\documentclass[12pt]{article}

\usepackage{cite, amsmath, amssymb}
\usepackage{color}
\newtheorem{theorem}{\bf Theorem}[section]
\newtheorem{corollary}[theorem]{\bf Corollary}

\newtheorem{conjecture}[theorem]{\bf Conjecture}

\newcommand{\proof}{\noindent{\bf Proof.\ }}
\newcommand{\qed}{\hfill $\square$ \bigskip}

\newcommand{\rad}{{\rm rad}}
\newcommand{\diam}{{\rm diam}}

\begin{document}

\title{On the Relation Between Wiener Index and Eccentricity of a Graph}

\author{
    Hamid Darabi $^{a}$
    \and
	Yaser Alizadeh $^{b}$
	\and
	Sandi Klav\v zar $^{c,d,e}$
	\and
	Kinkar Chandra Das $^{f}$
}

\date{\today}

\maketitle
% \vspace{-0.8 cm}
\begin{center}
$^a$ Department of Mathematics, Esfarayen University of Technology, Esfarayen, Iran \\
e-mail: {\tt darabi@esfarayen.ac.ir}\\
\medskip

$^b$ Department of Mathematics, Hakim Sabzevari University, Sabzevar, Iran \\
e-mail: {\tt y.alizadeh@hsu.ac.ir} \\
\medskip

$^c$ Faculty of Mathematics and Physics, University of Ljubljana, Slovenia\\
e-mail: {\tt sandi.klavzar@fmf.uni-lj.si} \\
\medskip

$^d$ Faculty of Natural Sciences and Mathematics, University of Maribor, Slovenia \\
\medskip

$^e$ Institute of Mathematics, Physics and Mechanics, Ljubljana, Slovenia\\
\medskip

$^f$ Department of Mathematics, Sungkyunkwan University, Suwon 16419, Republic of Korea \\
e-mail: {\tt kinkardas2003@googlemail.com} 
\end{center}

\begin{abstract}
The relation between the Wiener index $W(G)$ and the eccentricity $\varepsilon(G)$ of a graph $G$ is studied. Lower and upper bounds on $W(G)$ in terms of $\varepsilon(G)$ are proved and extremal graphs characterized. A Nordhaus-Gaddum type result on $W(G)$ involving $\varepsilon(G)$ is given. A sharp upper bound on the Wiener index of a tree in terms of its eccentricity is proved. It is shown that in the class of trees of the same order, the difference $W(T) - \varepsilon(T)$ is minimized on caterpillars. An exact formula for  $W(T) - \varepsilon(T)$ in terms of the radius of a tree $T$ is obtained. A lower bound on the eccentricity of a tree in terms of its radius is also given. Two conjectures are proposed. The first asserts that the difference $W(G) - \varepsilon(G)$ does not increase after contracting an edge of $G$. The second conjecture asserts that the difference between the Wiener index of a graph and its eccentricity is largest on paths. 
\end{abstract}

\noindent {\bf Key words:} Wiener index; eccentricity; eccentric connectivity; tree; extremal graph

% \medskip\noindent
% {\bf AMS Subj.\ Class:} 05C12, 92E10  

%%%%%%%%%%%%%%%%%%%%%%%%%%%%%%%%%%%%%%%
%%%%%%%%%%%%%%%%%%%%%%%%%%%%%%%%%%%%%%%
\section{Introduction}
\label{sec:intro}
%%%%%%%%%%%%%%%%%%%%%%%%%%%%%%%%%%%%%%%
%%%%%%%%%%%%%%%%%%%%%%%%%%%%%%%%%%%%%%%

All graphs in this paper are simple and connected. The order and the size of a graph $G$ will be denoted by $n(G)$ and $m(G)$, respectively. If $G = (V(G), E(G))$ is a graph and $u,v\in V(G)$, then the distance $d_G(u,v)$  is the number of edges on a shortest $u,v$-path. (By a $u,v$-path in $G$ we mean a path in $G$ whose end-vertices are the vertices $u$ and $v$.) The Wiener index of a graph $G$,
$$W(G)= \sum_{\{u,v\}\subseteq V(G)} d_G(u,v)\,,$$
is the oldest graph invariant (alias {\em topological index}) studied in mathematical chemistry~\cite{wiener-1947}. It is also one of the most studied among such indices, cf.\ the surveys~\cite{dobrynin-2001, dobrynin-2002, dobrynin-2012, knor-2016}, and continues to be an active research field~\cite{alizadeh-2018, dobrynin-2018, gutman-2017, gyori-2021, iran-2019, knor-2018, pan-2018, peterin-2018, tan-2018}. The {\em total distance of a vertex} $v$ of a graph $G$ is defined as $d_G(v)= \sum _{u \in V(G)}d_G(v,u)$.

If $v$ is a vertex of a graph $G$, then the {\em eccentricity} $\varepsilon_G(v)$ of a vertex $v$ is the distance from $v$ to a farthest vertex from $v$. A vertex $u$ is said to be an {\em eccentric vertex of $v$} if $d_G(v,u)=\varepsilon_G(v)$. The {\em radius} $\rad(G)$ of $G$ and the {\em diameter} $\diam(G)$ of $G$ are the minimum and the maximum eccentricity, respectively. The {\em center} $C(G)$ of $G$ is the set of vertices with minimum eccentricity, that is, $C(G) = \{u\in V(G):\ \varepsilon_G(u) = \rad(G)\}$. The {\em eccentricity} of a graph $G$ is
$$\varepsilon(G) = \sum_{v \in V(G)} \varepsilon_G(v)\,.$$
The eccentricity of a graph has been earlier studied on graph operations in~\cite{de-2015, fath-2014}, where the invariant was named {\em total eccentricity of a graph} but we believe that ``eccentricity of a graph" suffices because (i) this term is not used elsewhere and (ii) this is also consistent with the notation and terminology from~\cite{he-2018, hinz-2012}. The investigations of the Wiener index are in a way equivalent with the studies of the average distance. Similarly, the studies of the eccentricity are parallel with the research of the average eccentricity, the later being studied in particular in~\cite{dankelmann-2014, dankelmann-2020, das-2017, du-2013, he-2018, hinz-2012, ilic-2012, krnc-2020, tang-2012}. In~\cite{casablanca-2018}, the Wiener index has been studied on strong product graphs along with the average eccentricity. For a wider picture on eccentricity based descriptors for QSAR/QSPR we refer to~\cite{madan-2010}.

In this paper we are interested in the difference between the Wiener index and the eccentricity of a graph. In the next section we first give two lower bounds on $W(G)$ in terms of $\varepsilon(G)$ and then prove two related upper bounds. In all the cases we characterize the graphs that attain the bounds. In the last result of the section we prove a Nordhaus-Gaddum type result on the Wiener index of a graph involving its eccentricity. In Section~\ref{sec:trees} we concentrate on trees. First a sharp upper bound on the Wiener index of a tree in terms of its eccentricity is given. Then we
prove that the difference $W(T) - \varepsilon(T)$ on trees $T$ of the same order is minimized on caterpillars. We next give an exact formula for  $W(T) - \varepsilon(T)$ in terms of the radius of $T$. We also give a lower bound on the eccentricity of a tree in terms of the radius. (For results that relate the maximum Wiener index of trees with a given radius see~\cite{das-2017b}.) We conclude the paper with two conjectures. The first asserts that the difference $W - \varepsilon$ does not increase after contracting an edge. We support the conjecture by proving that it holds for the case when the contracted edge is a bridge. The second conjecture asserts that the difference between the Wiener index of a graph and its eccentricity is largest on paths. Before we begin with the results, some further definitions are given. 

Let $G$ be a graph. The degree of a vertex $u\in V(G)$ will be denoted with ${\rm deg}_G(u)$ or ${\rm deg}(u)$ for short. A {\em matching} of $G$ is a set of independent edges of $G$, that is, a set of edges no pair of them sharing an end-vertex. A matching $M$ is {\em perfect} if every vertex of $G$ is an end-vertex of some edge from $M$. We will use $K_n$, $P_n$, and $C_n$ to denote the complete graph of order $n$, the path of order $n$, and the cycle of order $n$, respectively. By $K_{n,m}$ we denote the complete bipartite graph with bipartition sets of order $n$ and $m$; in particular, $K_{1,m}$ is the star of order $m+1$. In the rest of the paper we may abbreviate $d_G(u,v)$, $d_G(v)$, $\varepsilon_G(v)$, and  ${\rm deg}_G(u)$ to $d(u,v)$, $d(v)$, $\varepsilon(v)$, and ${\rm deg}(u)$, respectively, when $G$ will be clear from the context.

%%%%%%%%%%%%%%%%%%%%%%%%%%%%%%%%%%%%%%%
%%%%%%%%%%%%%%%%%%%%%%%%%%%%%%%%%%%%%%%
\section{General graphs}
\label{sec:general-graphs}
%%%%%%%%%%%%%%%%%%%%%%%%%%%%%%%%%%%%%%%
%%%%%%%%%%%%%%%%%%%%%%%%%%%%%%%%%%%%%%%

\begin{theorem}
\label{thm:k1}
If $G$ is a graph, then
$$W(G)\geq \varepsilon(G)+m(G)-n(G)\,,$$
equality holding if and only if $G$ is obtained from $K_{n(G)}$ be removing a matching.
\end{theorem}

\proof
Set $n = n(G)$, $m=m(G)$, and let $S$ be the set of universal vertices of $G$, that is, $S = \{u\in V(G):\ {\rm deg}(u) = n-1\}$. Set further $k = |S|$.

\medskip\noindent
{\bf Case 1.} $k\ge 1$. \\
From the definition of the Wiener index we get that
\begin{equation}
\label{eq:wiener-lower}
2W(G) = \sum\limits_{v\in V(G)}\,d(v)\geq k(n-1)+(n-k)n=n^2-k\,.
\end{equation}
Moreover, equality holds in~\eqref{eq:wiener-lower} if and only if $\deg(v)=n-2$ for every vertex $v\in V(G)\backslash S$. This in turn holds if and only if $G$ is obtained from $K_{n}$ be removing a matching. (An easy way to verify this fact is to consider the complement of $G$.)

From the definition of the total eccentricity we infer that
\begin{equation}
\label{eq:zeta-equality}
\varepsilon(G) = \sum\limits_{v\in V(G)}\,\varepsilon(v)=k+(n-k)2=2n-k\,.
\end{equation}
Moreover, the Handshaking Lemma yields
\begin{equation}
\label{eq:Handshaking}
2m = \sum\limits_{v\in V(G)}\,\deg(v)\leq k(n-1)+(n-k)(n-2)=n^2-2n+k\,.
\end{equation}
Again, the equality in~\eqref{eq:Handshaking} holds if and only if $\deg(v)=n-2$ for every $v\in V(G)\backslash S$, that is, if and only if $G$ is obtained from $K_{n}$ be removing a matching.

From the above (in)equalities we get:
\begin{eqnarray*}
2W(G) & \stackrel{\eqref{eq:wiener-lower}}{\ge}  & n^2-k
\stackrel{\eqref{eq:Handshaking}}\ge (2m + 2n - k) - k  = 2m + 2n - 2k \\
& \stackrel{\eqref{eq:zeta-equality}} = & 2m + 2n - 2(2n - \varepsilon(G))  = 2\varepsilon(G) + 2m - 2n\,.
\end{eqnarray*}

\medskip\noindent
{\bf Case 2.} $k = 0$. \\
In this case $\varepsilon(v)\geq 2$ holds for every $v\in V(G)$ and consequently
\begin{eqnarray}
2W(G) = \sum\limits_{v\in V(G)}\,d(v) & \geq  & \sum\limits_{v\in V(G)}\,\Big(\varepsilon(v)+(\varepsilon(v)-1)+(n-3)\Big) \nonumber \\
& = & 2\varepsilon(G)+n(n-4)\,. \label{eq:k=0}
\end{eqnarray}
As $k=0$, the Handshaking Lemma implies $n(n-2)\ge 2m$ and thus $n(n-4)\ge 2m-2n$. Combining this fact with~\eqref{eq:k=0} we infer that $2W(G) \ge 2\varepsilon(G)+2m-2n$ holds also in this case. Moreover, the equality holds if and only if $d(v)=2\varepsilon(v)+n-4$ for every $v\in V(G)$ and $2m=n(n-2)$, that is, if and only if $G$ is the graph obtained from $K_n$ by removing a perfect matching.
\qed

With the help of Theorem~\ref{thm:k1} we can deduce the following result independent from the order and size of a graph considered.

\begin{theorem}
\label{thm:W-at-least-zeta}
If $G$ is a graph with $n(G)\ge 4$, then $W(G)\geq \varepsilon(G)$. Moreover, equality holds if and only if $G\in \{P_4, C_4\}$.
\end{theorem}

\proof
Again set $n = n(G)$ and $m=m(G)$. If $m\geq n+1$, then Theorem~\ref{thm:k1} immediately yields $W(G)>\varepsilon(G)$. Suppose next that $m=n$. Then Theorem~\ref{thm:k1} gives $W(G)\geq \varepsilon(G)$ with the equality if and only if $G$ is $K_n$ minus a matching, where $m=n$. If $n\geq 5$, then this is clearly not possible. Hence the only graph with $m=n$ that attains the equality is $C_4$. Since $G$ is connected, the only remaining case to consider is $m=n-1$ which means that $G$ is a tree.

So let $G$ be a tree. If $G\cong K_{1,\,n-1}$, then $W(G)>\varepsilon(G)$ as $n\geq 4$. Assume hence in the rest that $\varepsilon(v)\geq 2$ for every $v\in V(G)$. Then
\begin{eqnarray*}
 d(v) = \sum\limits_{v\in V(G)}\,d(u,\,v) & \geq & (1+2+\cdots+\varepsilon(v))+(n-1-\varepsilon(v)) \\
& = & n-1+\frac{1}{2}\,\Big(\varepsilon(v)^2-\varepsilon(v)\Big)\,.
\end{eqnarray*}
 Since $f(x)=2(n-1)+x^2-5x$ is an increasing function on $[3,\,n-1]$, for $x\geq 2$ ($x$ is an integer) we have
    $$f(x)\geq \min\{f(2),\,f(3)\}\geq 0~~\mbox{ as }~n\geq 4.$$
 From the above results we deduce
  $$2(n-1)+\varepsilon(v)^2-5\varepsilon(v)\geq 0,\mbox{ that is, }d(v)\geq n-1+\frac{1}{2}\,\Big(\varepsilon(v)^2-\varepsilon(v)\Big)\geq 2\varepsilon(v).$$
Therefore
 $$2W(G)=\sum\limits_{v\in V(G)}\,d(v)\geq 2\sum\limits_{v\in V(G)}\,\varepsilon(v)=2\,\varepsilon(G).$$
 Moreover, the equality holds if and only if $n=4$ and $\varepsilon(v)\in \{2, 3\}$ for any $v\in V(G)$, that is, if and only if $G\cong P_4$.
\qed

Note that Theorem~\ref{thm:W-at-least-zeta} implies that if $n(G)\geq 5$, then $W(G)>\varepsilon(G)$.

In Theorems~\ref{thm:k1} and~\ref{thm:W-at-least-zeta} we have bounded $W(G)$ from below using $\varepsilon(G)$. On the other hand, in~\cite{dankelmann-2004} it was observed (in terms of the average distance and average eccentricity) that $W(G)\leq \frac{n(G)-1}{2}\,\varepsilon(G)$ holds for any graph $G$. We add here that equality holds if and only if $G\cong K_{n(G)}$. To give further upper bounds on $W(G)$ using $\varepsilon(G)$, we need to recall a couple of concepts. First, the {\em eccentric connectivity} $\xi^c (G)$ of $G$ is
$$\xi^c (G) = \sum_{v \in V(G)}\deg(v)  \varepsilon(v)\,.$$
This graph invariant has been already well investigated, see the selection of related papers~\cite{das-2015, dankelmann-morgan, dankelmann-2014, doslic-2014, gupta-2000, sharma-1997, xu-2016, zhang-2019}. Second, $G$ is a {\em self-centered graph} if $\rad(G) = \diam(G)$, cf.~\cite{buckley-1989, janakiraman-2008}.

\begin{theorem}
\label{thm:from-above}
$(i)$ If $G$ is a graph, then $2W(G)\leq (n(G)-1)\,\varepsilon(G)-\xi^c(G)+2m(G)$, equality holding if and only if  $\diam(G) \le 2$.

\noindent
$(ii)$ If $G$ is a self-centered graph, then
$$W(G) \leq \varepsilon(G) + \left \{
\begin{array}{ll}
  \frac{n(n-2)^2}{8}; & n\ {\rm even}\,, \\[3mm]
  \frac{n[(n-2)^2-1]}{8}; & n\ {\rm odd}\,.
\end{array} \right.
$$
Moreover, equality holds if and only if $G$ is an odd cycle.
\end{theorem}

\proof
Throughout the proof let $n = n(G)$ and $m=m(G)$.

$(i)$ If $v\in V(G)$, then
 \begin{eqnarray}
  d(v)=\sum\limits_{u\in V(G)}\,d(v,u) &\leq&\deg(v)+(n-1-\deg(v))\,\varepsilon(v)\nonumber\\[2mm]
   &=&(n-1)\,\varepsilon(v)-\deg(v)\,(\varepsilon(v)-1).\nonumber
 \end{eqnarray}
Summing over the vertices of $G$ we get
$$2W(G) \le (n-1)\,\varepsilon(G)-\xi^c(G)+2m\,.$$
Moreover, equality holds if and only if $\varepsilon(v)\in \{1,2\}$, in other words, if and only if $\diam(G)\le 2$.

$(ii)$ Note first that considering a possible cut-vertex of a self-centered graph $G$ we infer that $G$ is $2$-connected. This means that $n\ge 3$. Whitney's theorem (which characterizes $2$-connected graphs) asserts that for every distinct vertices $v$ and $u$, there exist two internally disjoint $u,v$-paths. Hence, if $v\in V(G)$, then, observing a vertex $w$ at distance $d = \diam(G) = \rad(G)$ from $v$, we infer that
 \begin{eqnarray*}
 \frac{d(v)}{2}-\varepsilon(v)&\leq&[1+2+\cdots+(d-1)]+\frac{d(n-2d+1)}{2}-d\\
 &=&\frac{d(n-d-2)}{2}.
 \end{eqnarray*}
 Since $g(x)=x(n-x-2)$ is an increasing function on $x\leq \frac{n-2}{2}$, and a decreasing function on $x\geq \frac{n-2}{2}$, we have
$$g(x)\leq\left \{ \begin{array}{ll}
  \frac{(n-2)^2}{4}; & n\ {\rm even}\,, \\[3mm]
  \frac{(n-2)^2-1}{4}; & n\ {\rm odd}\,.
\end{array} \right.$$
Hence
$$\frac{d(v)}{2}-\varepsilon(v) \leq \left \{ \begin{array}{ll}
  \frac{(n-2)^2}{8}; & n\ {\rm even}\,, \\[3mm]
  \frac{(n-2)^2-1}{8}; & n\ {\rm odd}\,.
\end{array} \right.$$
\noindent
Summing over all vertices we get
 $$W(G)-\varepsilon(G)\leq\left \{ \begin{array}{ll}
  \frac{n(n-2)^2}{8}~~~&\mbox{ when $n$ is even}, \\[5mm]
  \frac{n[(n-2)^2-1]}{8}~~~&\mbox{ when $n$ is odd.}
 \end{array} \right.$$
Moreover, it is easily seen that equality holds if and only if $G$ is an odd cycle.
\qed

As a consequence of Theorem~\ref{thm:from-above} we have the following Nordhaus-Gaddum type result (cf.~\cite{li-2011, mao-2017}). 

\begin{corollary}
If both $G$ and its complement $\Bar{G}$ are connected, then
$$ W(G) + W(\Bar{G}) \le \binom{n(G)}{2} + \frac{1}{2}\left[(n(G)-1) (\varepsilon(G) + \varepsilon(\bar{G})) - \xi^c(\Bar{G}) - \xi^c(G)\right]\,.$$
\end{corollary}

\proof
Summing the inequalities for $G$ and $\Bar{G}$ expressed by Theorem~\ref{thm:from-above} and using the facts $n(\Bar{G}) = n(G)$ and $m(G) + m(\Bar{G}) = \binom{n(G)}{2}$, the result follows. 
\qed

%%%%%%%%%%%%%%%%%%%%%%%%%%%%%%%%%%%%%%%
%%%%%%%%%%%%%%%%%%%%%%%%%%%%%%%%%%%%%%%
\section{Trees}
\label{sec:trees}
%%%%%%%%%%%%%%%%%%%%%%%%%%%%%%%%%%%%%%%
%%%%%%%%%%%%%%%%%%%%%%%%%%%%%%%%%%%%%%%

An upper bound on the Wiener index in terms of eccentric connectivity has been reported in \cite{dankelmann-morgan}. Here we determine a sharp upper bound on the Wiener index of a tree in terms of its eccentricity.

 \begin{theorem} If $T$ is a tree of order $n$, then
 \begin{equation*}
 W(T)\leq \frac{1}{4}\,(2n-3)\varepsilon(T) + \frac{1}{4}
 \end{equation*}
 with equality holding if and only if $T$ is the star $K_{1,\,n-1}$.
 \end{theorem}

 \begin{proof} If $T = K_{1,\,n-1}$, then 
 $$W(K_{1,\,n-1})=(n-1)^2=\frac{1}{4}\,(2n-3)\varepsilon(K_{1,\,n-1}) + \frac{1}{4}\,,$$
hence the equality holds. If $T$ is not a star, then $\varepsilon(v)\geq 2$ for every vertex $v\in V(T)$. Hence,  $$\left(n-\frac{3}{2}\right)\,\varepsilon(v)=(n-2)\,\varepsilon(v)+\frac{1}{2}\,\varepsilon(v)\geq (n-2)\,\varepsilon(v)+1\geq \sum\limits_{u\in V(T)}\,d(v,\,u)=d(v),$$
 that is,
 $$(2n-3)\,\varepsilon(v)\geq 2d(v).$$
Summing over all vertices of $T$ we get $4W(T)\leq (2n-3)\,\varepsilon(T)$, that is, $4W(T)< (2n-3)\,\varepsilon(T)+1$. 
 \end{proof} \qed
 
We next show that the minimum difference between the Wiener index and the eccentricity of a tree is achieved on caterpillars. Recall that a tree is a {\em caterpillar} if it contains a (diametrical) path, such that any vertex not on the path is at distance $1$ from it.
\begin{theorem}
If $n$ is a positive integer, then $\min\{W(T) - \varepsilon(T):\ T\ {\rm tree},\ n(T) = n\}$ is achieved on a caterpillar.
\end{theorem}

\proof
Let $x$ and $y$ be diametrical vertices of $T$ and let $P$ the $x,y$-path in $T$. Since $P$ contains the center of $T$, for each vertex $w \in T$, the eccentricity of $w$ is equal to $d(w,x)$ or $d(w,y)$.

Suppose that $T$ is not a caterpillar. This is equivalent to the fact that in $V(T)\setminus V(P)$ there exists at least one vertex which is not of degree $1$. Among all such vertices select a vertex $u$ which is farthest from $P$. That is, $u$ is a vertex from $V(T)\setminus V(P)$ of degree at least $2$, such that $d_T(u,P) = \min\{d(u,v):\ v\in V(P)\}$ is largest. Let $z\in V(P)$ be the unique vertex of $P$ which is closest to $u$, that is, $d_T(u,z) = d_T(u,P)$. Let $d_T(u,z) = \ell$ and note that since $T$ is not a caterpillar, $\ell\ge 1$. Let $v$ be the neighbor of $u$ on the $u,z$-path, and let $T_z$ be the maximal subtree of $T$ that contains $z$ and no other vertex of $P$. Clearly, $\varepsilon_{T_z}(z) = \ell+1\ge 2$. If $\ell=1$, then vertices 
$v$ and $z$ are the same vertex.

Consider the following transformation. Let $u$ and $v$ be two adjacent vertices of $T_z$ which $d(z,u)=\ell$ and $d(z,v)=\ell-1$. Let $S = \{w:\ uw\in E(T)\} \setminus \{v\}$ and $|S|=s\geq 1$. Assume that $T'$ is a tree obtained from $T$ by removing the edges between $u$ and the vertices of $S$ and then connecting vertex $v$ to the vertices of $S$. It is clear that the distances between the vertices of $V(G) \setminus S$ are the same in $T$ and $T'$. Also the distances of vertices of $S$ from the other vertices, except $u$, decrease by 1, and for each vertex $w\in S$ we have  $d_{T'}(u,w)=2, d_{T}(u,w)=1$. Therefore, 
\begin{eqnarray*}
W(T) - W(T') &=& \sum _{\substack { j\in S \\ w\in V(T)-S}}\,\Big[d_{T}(j,w) - d_{T'}(j,w)\Big] \\[3mm] &=&\sum_{j \in S}\,\Big[d_T(j,u)-d_{T'}(j,u)\Big]+\sum_{\substack {j \in S \\ u\neq w \in V(T)-S}}  \,\Big[d_T(j,w)-d_{T'}(j,w)\Big] \\[3mm]
&=& -s +s(n-s-1)=sn-s^2-2s.
\end{eqnarray*}
By the above transformation, each vertex $w \in V(T) \setminus S$ has the same eccentricity in $T$ and in $T'$,  and the eccentricity of the vertices from $S$ decreases by $1$. Hence $\varepsilon(T) - \varepsilon(T') =s$. It follows that
$$W(T) -W(T') = (sn-s^2-2s) >s = \varepsilon(T) -\varepsilon(T')\,,$$
and we are done.
\qed

We next give a formula for $W(T) - \varepsilon(T)$ for a tree $T$ of a given radius. For this sake recall that the {\em line graph}, $L(G)$, of a graph $G$ has the vertex set $V (L(G)) = E(G)$ and two distinct vertices of $L(G)$ are adjacent if the corresponding edges of $G$ share a common end-vertex. Buckley~\cite{Buckley} observed the following simple relation between the Wiener index of a tree $T$ and of its line graph:

\begin{equation}
\label{eq:win}
W(T) = W(L(T)) + \binom{n}{2}\,.
\end{equation}

For the rest of the section we recall that if $T$ is a tree, then its center $C(T)$ consists either of a single vertex or of two adjacent vertices. This fact can be, for instance, deduced by iteratively removing all the leaves of a tree considered, until the center is found. Using~\eqref{eq:win} we next prove the following result.

\begin{theorem}
If $T$ is a tree, $r = \rad(T)$, and $n = n(T)$, then
$$W(T) - \varepsilon(T) = W(L(T)) - \varepsilon(L(T)) + \frac{n(n-3)}{2} -r +1\,.$$
\end{theorem}

\proof
First we are going to find a relation between $\varepsilon(T)$ and $\varepsilon(L(T))$. For a vertex $v$, denote by $d(v,C(T))$ the minimum distance between $v$ and central vertices of $T$. For each non-central vertex $v$, there is a unique adjacent vertex $w$ such that $d(v,C(T)) = d(w,C(T)) + 1$.
Consider the bijection $f: V(T)\setminus C(T) \rightarrow E(T) \setminus E(C(T))$, where $f(v) = vw$. It is not difficult to see that $\varepsilon_{T}(v) = \varepsilon_{L(T)}(f(v)) + 1$. We consider two cases.

Suppose first that $C(T)=\{p\}$. Then
\begin{eqnarray*}
\varepsilon(T) & = & r + \sum_{v\in V(T)-C(T)} \varepsilon_T(v) \\[2mm]
  & = & r + \sum_{v\in V(T)-C(T)}\,\Big[\varepsilon_L(T)(vw)+1\Big]\\[2mm]
  & = & r+n-1 + \varepsilon(L(T))\,.
\end{eqnarray*}
In the second case assume that $C(T)=\{p, q\}$. Then
\begin{eqnarray*}
\varepsilon(T) & = & 2r + \sum_{v\in V(T),\,v \neq p,\,q}\varepsilon(v) \\[3mm]
  & = & 2r + \sum_{v\in V(T),\,v \neq p,\,q}\,\Big[\varepsilon(f(v)) +1\Big] \\[3mm]
  & = & 2r + \varepsilon(L(T)) - (r-1) + n-2 \\
  & = & \varepsilon(L(T)) + n+r-1\,.
\end{eqnarray*}
We have thus seen that in each of the cases $\varepsilon(T) = \varepsilon(L(T)) + n+r-1$ holds. The assertion of the theorem follows by combining this relation with~\eqref{eq:win}.
\qed

From~\cite[Theorem 2]{dankelmann-2004}, which is stated in terms of the average distance and average eccentricity, we extract that if $T$ is a tree, $n = n(T)$, and $r = \rad(T)$, then
\begin{equation}
\label{prop:ecc}
\varepsilon(T)=
\begin{cases}
 d_T(p) +  nr; & C(T) = \{p\}, \\[3mm]
 \displaystyle{\frac{d_T(p) + d_T(q) -n}{2}} +  nr; & C(T)=\{p,q\}\,.
\end{cases}
\end{equation}
In the last result of the section we apply~\eqref{prop:ecc} to give a lower bound on the eccentricity of a tree in terms of its radius.

\begin{theorem}
If $T$ is a tree, $n = n(T)$, and $r = \rad(T)$, then
$$ \varepsilon   \big(T\big)   \geq
\begin{cases}
r(n+r+1); & |C(T)|=1, \\ \\
\displaystyle{\frac{2r(n+r)-n}{2}}; & |C(T)|=2\,.
\end{cases}
$$
Moreover, the equality holds if and only if $T$ is a path.
\end{theorem}

\proof
We again consider two cases based on the cardinality of the center of $T$. 

\medskip\noindent
\textbf{Case 1}: $C(T)= \left\{ p \right\}$.\\
In this case, for any integer $t$, where $1 \le t \le r$, there are at least two vertices at distance $t$ from $p$. Thus $d_T(p) \ge 2(1+2+ \cdots +r) = r(r+1)$.

\medskip\noindent
\textbf{Case 2}: $C(T)= \left\{ p ,q \right\}$.\\
Analogously, for any integer $1\le t \le r-1$, there are at least two vertices of distance $t$ and a vertex of distance $r$ from central vertex $p$. Then
$d_T(p) \ge 2(1+2+ \cdots + r-1) + r = r^2$.

Note that in any of the two cases, the equality holds if and only if $T$ is a path.
The result now follows from~\eqref{prop:ecc}.
\qed

%%%%%%%%%%%%%%%%%%%%%%%%%%%%%%%%%%%%%%%
%%%%%%%%%%%%%%%%%%%%%%%%%%%%%%%%%%%%%%%
\section{Two conjectures}
\label{sec:conclude}
%%%%%%%%%%%%%%%%%%%%%%%%%%%%%%%%%%%%%%%
%%%%%%%%%%%%%%%%%%%%%%%%%%%%%%%%%%%%%%%

If $e$ is an edge of a graph $G$, then let $G.e$ denote the graph obtained from $G$ by contracting the edge $e$. Our first conjecture asserts the following.

\begin{conjecture}
\label{conj:contract}
If $e$ is an edge of a graph $G$ with $n(G)\ge 3$, then
$$W(G.e) - \varepsilon(G.e) \le W(G) - \varepsilon(G)\,.$$
\end{conjecture}

The next result is a partial support for the conjecture.

\begin{theorem}
If $e$ is a bridge of a graph $G$ with $n(G)\ge 3$, then
$$W(G.e) - \varepsilon(G.e) \le W(G) - \varepsilon(G)\,.$$
\end{theorem}

\proof
Let $e=uv$ and let $G -uv = G_u \cup G_v$ where $G_u$ and $G_v$ are the components of $G -uv$ containing $u$ and $v$ respectively. Then
\begin{eqnarray*}
W(G) &=& W(G_u) + W(G_v) + \sum_{x\in G_u}\sum_{y\in G_v} d(x,y) \\ &=& W(G_u) + W(G_v) + \sum_{x\in G_u}\sum_{y\in G_v} d(x,u)+1+d(v,y)\\ &=&
W(G_u) + W(G_v) + n(G_v)d_{G_u}(u) + n(G_u)n(G_v)+ n(G_u)d_{G_v}(v)\,.
\end{eqnarray*}
Since $e=uv$ is a bridge, for every $x\in V(G_u)$ and every $y\in V(G_v)$ we have $d_{G}(x, y) = d_{G}(x, u) + 1 + d_{G}(v, y)$. This in turn implies that $d_{G.e}(x, y) = d_{G}(x, y) - 1$. Therefore, in $G.e$ we have
\begin{eqnarray*}
W(G.e) &=& W(G_u) + W(G_v) + \sum_{x\in G_u}\sum_{y\in G_v} d(x,y) \\ &=&
W(G_u) + W(G_v)+ (n(G_v)-1) d_{G_u}(u) + (n(G_u)-1)d_{G_v}(v).
\end{eqnarray*}
Hence
$$W(G) - W(G.e) = d_{G_u}(u) + d_{G_v}(v) + n(G_u)n(G_v)$$
Let $n(G_v)\le n(G_u)$. If $n(G_v) =1$, then the eccentricity of the vertices from $G_u-u$ decreases by at most $1$. Note further that the eccentricity of $u$ does not change. Hence
$$\varepsilon(G) - \varepsilon(G.e) \le n(G_u)-1 + n(G_u) = 2n(G_u)-1.$$
Then
$$ W(G) -W(G.e) \ge n(G_u)-1 + n(G_v)-1 + n(G_u) \ge \varepsilon(G) - \varepsilon(G.e)\,.$$
In the case when $n(G_v)\ge 2$, then we get
\begin{eqnarray*}
W(G) -W(G.e) &\ge& n(G_u)-1 +n(G_v) -1 + n(G_u)n(G_v) \\ &\ge & n(G_u)+n(G_v)-2 + n(G_u) + n(G_v) =2n(G)-2.
\end{eqnarray*}
The eccentricity of the vertices from $V(G.e)$ decreases by at most $1$ and the eccentricity of the removed vertex by at most $n(G)-1$. Then $\varepsilon(G) -\varepsilon(G.e) \le 2n(G) -2 \le W(G) -W(G.e)$. Therefore, the difference of the Wiener index of two graphs is greater than or equal to the difference of their total eccentricities.
\qed

In case Conjecture~\ref{conj:contract} holds true, it cannot be extended to all minors $H$ of $G$, because the same property does not hold for removing edges. For a simple example consider the paw graph $G$ (a graph obtained by adding a pendant vertex to a triangle), and let $H=G-e$, where $e$ is the edge of $G$ with both end-vertices of degree $2$, that is $H = K_{1,3}$. Then $W(H) - \varepsilon(H) = 9 - 7 = 2 > 1 = 8 - 7 = W(G) - \varepsilon(G)$.

Our second conjecture asserts that the difference between the Wiener index of a graph and its eccentricity is largest possible on paths. More precisely:

\begin{conjecture}
If $G$ is a graph of order $n$ with $\rad(G) \ge 4$, then
$$W(G)-\varepsilon(G) \le  \left\lfloor \frac{1}{6}n^3 -  \frac{3}{4}n^2 + \frac{1}{3}n + \frac{1}{4} \right\rfloor$$
with equality holding if and only if $G$ is a path.
\end{conjecture}

The condition $\rad(G) \ge 4$ is posed because otherwise the equality case is achieved also by graphs different from paths. For instance, let $T_7$ be the tree obtained from $P_6$ by adding one new vertex and connecting it with an edge with the second vertex of $P_6$. Then $\rad(T_7) = 3$ and $W(T_7)-\varepsilon(T_7) = 52 - 29 = 23$, which is the equality case in the above expression. Similarly, if $T_8$ is the tree obtained from $P_7$ by adding an extra vertex adjacent with the second vertex of $P_7$, then $\rad(T_7) = 3$ and $W(T_7)-\varepsilon(T_7) = 79 - 39 = 40$, again the equality case in the above expression.

\section*{Acknowledgments}

S.K.\ acknowledges the financial support from the Slovenian Research Agency (research core funding P1-0297 and projects J1-9109, J1-1693,  N1-0095, N1-0108). K.C.\ Das was supported by the National Research Foundation of the Korean government with grant No. 2017R1D1A1B03028642.


\begin{thebibliography}{10}

\bibitem{alizadeh-2018}
  Y.~Alizadeh, S.~Klav\v{z}ar,
  On graphs whose Wiener complexity equals their order and on Wiener index of asymmetric graphs,
  {\em Appl.\ Math.\ Comput.} {\bf 328} (2018) 113--118.

\bibitem{Buckley}
  F.~Buckley,
  Mean distance in line graphs,
  {\em Congr.\ Numer.} {\bf 32} (1981) 153--162.

\bibitem{buckley-1989}
  F.~Buckley,
  Self-centered graphs.
  {\em Ann.\ New York Acad.\ Sci.} {\bf 576} (1989) 71--78.

\bibitem{casablanca-2018}
  R.M.~Casablanca, P.~Dankelmann,
Distance and eccentric sequences to bound the Wiener index, Hosoya polynomial and the average eccentricity in the strong products of graphs,
  {\em Discrete Appl.\ Math.} 263 (2019) 105--117.
  
\bibitem{dankelmann-2004}
  P.~Dankelmann, W.~Goddard, H.C.~Swart,
  The average eccentricity of a graph and its subgraphs,
  {\em Util.\ Math.} {\bf 65} (2004) 41--51.

\bibitem{dankelmann-morgan}
  P.~Dankelmann, M.J.~Morgan, S.~Mukwembi, H.C.~Swart,
  On the eccentric connectivity index and Wiener index of a graph,
  {\em Quaest.\ Math.} {\bf 37} (2014) 39--47.

\bibitem{dankelmann-2014}
  P.~Dankelmann, S.~Mukwembi,
  Upper bounds on the average eccentricity,
  {\em Discrete Appl.\ Math.} {\bf 167} (2014) 72--79.

\bibitem{dankelmann-2020}
  P.~Dankelmann, F.J.~Osaye,
  Average eccentricity, minimum degree and maximum degree in graphs,
  {\em J.\ Comb.\ Optim.} {\bf 40} (2020) 697--712.

\bibitem{das-2017}
  K.C.~Das, A.D.~Maden, A. Dilek, I.N.~Cang\"ul, A.S.~\c{C}evik,
  On average eccentricity of graphs,
  {\em Proc.\ Nat.\ Acad.\ Sci.\ India Sect.\ A} {\bf 87} (2017) 23--30.

\bibitem{das-2015}
  K.C.~Das, M.J.~Nadjafi-Arani,
  Comparison between the Szeged index and the eccentric connectivity index,
  {\em Discrete Appl.\ Math.} {\bf 186} (2015) 74--86.

\bibitem{das-2017b}
  K.C.~Das, M.J.~Nadjafi-Arani,
  On maximum Wiener index of trees and graphs with given radius,
  {\em J.\ Comb.\ Optim.} {\bf 34} (2017) 574--587.

\bibitem{de-2015}
  N.~De, Sk.Md.~Abu Nayeem, A.~Pal,
  Total eccentricity index of the generalized hierarchical product of graphs,
  {\em Int.\ J.\ Appl.\ Comput.\ Math.} {\bf 1} (2015) 503--511.

\bibitem{dobrynin-2018}
  A.A.~Dobrynin,
  The Szeged and Wiener indices of line graphs,
  {\em MATCH Commun.\ Math.\ Comput.\ Chem.} {\bf 79} (2018)  743--756.

\bibitem{dobrynin-2001}
  A.A.~Dobrynin, R.~Entringer, I.~Gutman,
  Wiener index of trees: theory and applications,
  {\em Acta Appl.\ Math.} {\bf 66} (2001) 211--249.

\bibitem{dobrynin-2002}
  A.A.~Dobrynin, I.~Gutman, S.~Klav\v{z}ar, P.~\v{Z}igert,
  Wiener index of hexagonal systems,
  {\em Acta Appl.\ Math.} {\bf 72} (2002) 247--294.

\bibitem{dobrynin-2012}
  A.A.~Dobrynin, L.S.~Mel'nikov,
  Wiener index of line graphs,
  in I. Gutman, B. Furtula (Eds.) Distance in Molecular Graphs - Theory,
  Univ. Kragujevac, Kragujevac (2012) 85--121.

\bibitem{doslic-2014}
  T.~Do\v sli\' c, M.~Saheli,
  Eccentric connectivity index of composite graphs,
  {\em Util.\ Math.} {\bf 95} (2014) 3--22.

\bibitem{du-2013}
  Z.~Du, A.~Ili\'c,
  On AGX conjectures regarding average eccentricity,
  {\em MATCH Commun.\ Math.\ Comput.\ Chem.} {\bf 69} (2013) 597--609.

\bibitem{fath-2014}
  K.~Fathalikhani, H.~Faramarzi, H.~Yousefi-Azari,
  Total eccentricity of some graph operations,
  {\em Electron. Notes Discrete Math.} {\bf 45} (2014) 125--131.

\bibitem{gupta-2000}
  S.~Gupta, M.~Singh, A.K.~Madan,
  Connective eccentricity index: a novel topological descriptor for predicting biological activity,
  {\em J. Mol. Graph. Model.} {\bf 18} (2000) 18--25.

\bibitem{gutman-2017}
  I.~Gutman, S.~Li, W.~Wei,
  Cacti with {$n$}-vertices and {$t$} cycles having extremal Wiener index,
  {\em Discrete Appl.\ Math.} {\bf 232} (2017) 189--200.

\bibitem{gyori-2021}
  E.~Gy\H{o}ri, A.~Paulos, C.~Xiao, 
  Wiener index of quadrangulation graphs,
  {\em Discrete Appl.\ Math.} {\bf 289} (2021) 262--269.

\bibitem{he-2018}
  C.~He, S.~Li, J.~Tu,
  Edge-grafting transformations on the average eccentricity of graphs and their applications,
  {\em Discrete Appl.\ Math.} {\bf 238} (2018) 95--105.

\bibitem{hinz-2012}
  A.M.~Hinz, D.~Parisse,
  The average eccentricity of {S}ierpi\'nski graphs,
  {\em Graphs Combin.} {\bf 28} (2012) 671--686.

\bibitem{ilic-2012}
  A.~Ili\'c,
  On the extremal properties of the average eccentricity,
  {\em Comput.\ Math.\ Appl.} {\bf 64} (2012) 2877--2885.

\bibitem{iran-2019}
  M.A.~Iranmanesh, H.~Shabani, 
  The symmetry-moderated Wiener index of truncation graph, Thorn graph and caterpillars, 
  {\em Discrete Appl.\ Math.} {\bf 269} (2019) 41--51.

\bibitem{janakiraman-2008}
  T.N.~Janakiraman, M.~Bhanumathi, S.~Muthammai,
  Self-centered super graph of a graph and center number of a graph,
  {\em Ars Combin.} {\bf 87} (2008) 271--290.

\bibitem{knor-2018}
  M.~Knor, S.~Majstorovi\'c, R.~\v Skrekovski,
  Graphs whose {W}iener index does not change when a specific vertex is removed,
  {\em Discrete Appl.\ Math.} {\bf 238} (2018) 126--132.

\bibitem{knor-2016}
  M.~Knor, R.~\v{S}krekovski, A.~Tepeh,
  Mathematical aspects of Wiener index,
  {\em Ars Math.\ Contemp.} {\bf 11} (2016) 327--352.

\bibitem{krnc-2020}
  M.~Krnc, J.-S.~Sereni, R.~\v{S}krekovski, Z.B.~Yilma,
  Eccentricity of networks with structural constraints,
  {\em Discuss.\ Math.\ Graph Theory} {\bf 40} (2020) 1141--1162.

\bibitem{li-2011}
  D.~Li, B.~Wu, X.~Yang, X.~An,
  Nordhaus-{G}addum-type theorem for {W}iener index of graphs when decomposing into three parts,
  {\em Discrete Appl.\ Math.} {\bf 159} (2011) 1594--1600.

\bibitem{madan-2010}
  A.K.~Madan, H.~Dureja,
  Eccentricity based descriptors for QSAR/QSPR,
  in: {\em Novel Molecular Structure Descriptors - Theory and Applications II},
  I.~Gutman, B.~Furtula (Eds.),
  Univ.\ Kragujevac, Kragujevac, 2010, 91--138.

\bibitem{mao-2017}
  Y.~Mao, Z.~Wang, I.~Gutman, H.~Li,
  Nordhaus--{G}addum-type results for the {S}teiner {W}iener index of graphs,
  {\em Discrete Appl.\ Math.} {\bf 219} (2017) 167--175.

\bibitem{pan-2018}
  J.-J.~Pan, C.-H.~Tsai,
  A lower bound for the {$t$}-tone chromatic number of a graph in terms of Wiener index,
  {\em Graphs Combin.} {\bf 34} (2018) 159--162.

\bibitem{peterin-2018}
  I.~Peterin, P.~\v Zigert Pleter\v sek,
  Wiener index of strong product of graphs,
  {\em Opuscula Math.} {\bf 38} (2018) 81--94.

\bibitem{sharma-1997}
  V.~Sharma, R.~Goswami, A.K.~Madan,
  Eccentric connectivity index: a novel highly discriminating topological descriptor for structure property and structure activity studies,
  {\em J.\ Chem.\ Inf.\ Comput.\ Sci.} {\bf 37} (1997) 273--282.

\bibitem{tan-2018}
  S.W.~Tan,
  The minimum {W}iener index of unicyclic graphs with a fixed diameter,
  {\em J.\ Appl.\ Math.\ Comput.} {\bf 56} (2018) 93--114.

\bibitem{tang-2012}
  Y.~Tang, B.~Zhou,
  On average eccentricity,
  {\em MATCH Commun.\ Math.\ Comput.\ Chem.} {\bf 67} (2012) 405--423.

\bibitem{wiener-1947}
  H.~Wiener,
  Structural determination of paraffin boiling points,
  {\em J.\ Amer.\ Chem.\ Soc.} {\bf 69} (1947) 17--20.

\bibitem{xu-2016}
  K.~Xu, K.C.~Das, H.~Liu,
  Some extremal results on the connective eccentricity index of graphs,
  {\em J.\ Math.\ Anal.\ Appl.} {\bf 433} (2016) 803--817.

\bibitem{zhang-2019}
  M.~Zhang, S.~Li, B.~Xu, G.~Wang, 
  On the minimal eccentric connectivity indices of bipartite graphs with some given parameters,
  {\em Discrete Appl.\ Math.} {\bf 258} (2019) 242--253.
\end{thebibliography}
\end{document}